\documentclass[11pt,letterpaper]{amsart}
\usepackage{hyperref}
\usepackage{amsfonts,mathrsfs,latexsym,amsmath,amssymb,amsthm,epsfig}
\usepackage{setspace, color}
\usepackage{tikz}

\newtheorem{thm}{Theorem}[section]
\newtheorem{cor}[thm]{Corollary}

\newtheorem{lem}[thm]{Lemma}
\newtheorem{prop}[thm]{Proposition}

\numberwithin{equation}{section}

\newcommand{\al}{\alpha}

\newcommand{\de}{\delta}
\newcommand{\De}{\Delta}
\newcommand{\ep}{\varepsilon}

\newcommand{\Om}{\Omega}
\newcommand{\ga}{\gamma}

\newcommand{\R}{\mathbb{R}}
\newcommand{\U}{\mathbb{S}}

\DeclareMathOperator{\vol}{vol}

\newcommand{\Real}{\mathbb{R}}

\newcommand{\norm}[1]{\Vert#1\Vert}

\def\<{\left\langle} \def\>{\right\rangle}
\def\({\left(} \def\){\right)}
\newcommand{\n}{\nabla}
\newcommand{\p}{\partial}
\def\s{\,\,\,\,}

\title[Huber's theorem]{Huber's theorem for manifolds with $L^\frac{n}{2}$ integrable
Ricci curvatures}

\author[B. Chen]{Bo Chen}
\address{Department of Mathematical Sciences, Tsinghua University, Beijing, 100084, China}
\email{chenbo@mail.tsinghua.edu.cn}

\author[Y. Li]{Yuxiang Li}
\address{Department of Mathematical Sciences, Tsinghua University, Beijing, 100084, China}
\email{liyuxiang@mail.tsinghua.edu.cn}
\date{}

\thanks{The first author is partially supported by China Postdoctoral Science Foundation, Grant No. 2021M701930. The second author is partially supported by NSFC 11971451 and NSFC 12141103.}
\dedicatory{\it Dedicated to Prof. Ernst Kuwert for his sixtieth birthday}
\begin{document}
\maketitle

\begin{abstract}
In this paper, we generalize Huber's finite points conformal compactification theorem to 
higher dimensional manifolds, which are conformally compact with $L^\frac{n}{2}$ integrable Ricci curvatures.
\end{abstract}
\maketitle
\section{Introduction}\label{s: intr}

A classical theorem of Huber states that a complete surface, whose negative part of the Gauss curvature is integrable, is conformally equivalent to a compact surface with a finite number of points removed.
By Cohn-Vossen's Theorem, the Gauss curvature is integrable on such a surface.

Naturally, one may wonder what is  a higher dimensional generalization of Huber's theorem.
For a Riemannian manifold $(M,g)$ with $\dim M\geq 3$, we have a number of choices of curvatures, such as scalar curvature $R_g$, Ricci curvature $Ric_g$, sectional curvature $K_g$, and $Q$-curvature $Q_g$.
However, all of the above curvatures vanish on the manifold $M=\mathbb{T}^{n-1}\times \R$, which is obviously not conformally equivalent to
a closed manifold with finitely many points removed. Hence, a Huber type
theorem is not true without additional assumptions. 

So far, many known results concerning Huber's theorems are proved on conformally compact manifolds, i.e., manifolds which are conformal domains of closed manifolds. For example, S. A. Chang, J. Qing, and P. Yang \cite{CQY} proved that, if $(M,g)$ is conformally equivalent to a domain of $\U^4$, and if  
$$
0<C_1<R_g<C_2, \s |\nabla R_g|<C_3, \s Ric_g\geq C_4g,\s
\int_M|Q_g|dV_g<+\infty,
$$ 
for some given constants $C_i$ with $i=1,\cdots,4$, then $\U^4\setminus M$ is a finite set. In the recent paper \cite{MQ}, S. Ma and J. Qing adopted the $n$-Laplace equation to study conformal metrics on a locally conformally flat manifold $(M,g)$. They assume that  either $Ric_g\geq 0$ outside a compact subset or 
$$
Ric_g^{-}\in L^\infty\cap L^1,\s R_g\in L^\infty, \s and\s
\nabla R_g\in L^\infty.
$$
First, they proved that a Huber-type theorem holds when $(M,g)$ is conformal to a domain of $\U^n$ with $n>2$. Then they showed that if  there exists a conformal immersion $\Phi:(M,g)\rightarrow \U^n$ with $n\geq5$, then $\Phi$ is an embedding and 
$\U^n\setminus \Phi(M)$ is a finite set.

In this paper, we will focus on conformally compact manifolds with bounded $L^\frac{n}{2}$-norm of Ricci curvatures. In this direction, the best result we know is due to Carron and Herzlich \cite{CH}.

\begin{thm}[Carron-Herzlich]
Let $\Omega$ be a domain of $(M,g_0)$, a compact Riemannian manifold of
dimension $n > 2$. Assume $\Omega$ is endowed with a complete Riemannian metric $g$
which is conformal to $g_0$. Suppose moreover that

-either the Ricci tensor of $g$ is in $L^\frac{n}{2}(\Omega,g)$ and
\begin{equation}\label{CHV}
\vol_g(B(x_0, r))=
O(r^n\log^{n-1}r)
\end{equation}
for some point $x_0$ in $\Omega$;

-or the positive part $R_g^{+}$ of the scalar curvature of $g$ is in $L^\frac{n}{2}(\Omega,g)$ and, for some point $x_0$ in $\Omega$, $\vol(B(x_0, r),g)=O(r^n)$.

Then there exists finitely many points $\{p_1,p_2,\cdots,p_N\}$ in $M$ such that
$$
\Om=M\setminus\{p_1,p_2,\cdots,p_N\}.
$$
\end{thm}

\begin{thm}[Carron-Herzlich]\label{CH1}
Let $\Omega$ be a domain of the sphere $(\mathbb{S}^4, g_{\mathbb{S}^4})$ endowed with a complete
metric $g$ conformal to $g_{\mathbb{S}^4}$. Assume moreover $Ric_g\in L^2(\Omega, g)$ and
\begin{equation}\label{noncritical}
\s R_g\in L^\frac{4}{3}\cap L^{2(1+\delta)}(\Omega,g)
\mbox{ for some }\delta\in(0,1).
\end{equation}
Then there is a finite set $\{p_1,\cdots,p_N\}$ such that $\Omega = \mathbb{S}^4\setminus\{p_1,...,p_N\}$.
\end{thm}

Recently,  Carron \cite{C20} obtained a new volume estimate
under some integral bounds of the negative part of the Ricci curvature and improved Theorem \ref{CH1} to the following version.

\begin{thm}[Carron]
Let $\Omega$ be a domain of $(M,g_0)$, a compact Riemannian manifold of
dimension $n > 2$. Assume $\Omega$ is endowed with a complete Riemannian metric $g$
which is conformal to $g_0$. Suppose moreover that for some $\nu>\frac{n}{2}$,
$$
\int_\Omega|Ric_g|^\frac{n}{2}dV_g<+\infty
$$
and
\begin{equation}\label{Ric-}
\int_\Omega|Ric^{-}_g|^\nu dV_g<+\infty. 
\end{equation}
Then there exists finitely many points $\{p_1,p_2,\cdots,p_N\}$ in $M$ such that
$$
\Om=M\setminus\{p_1,p_2,\cdots,p_N\}.
$$
\end{thm}

The main purpose of this paper is to show that \eqref{CHV}, \eqref{noncritical} and \eqref{Ric-} in the above theorems can be removed. Our
result
is stated as follows.
\begin{thm}\label{main}
Let $\Om$ be a domain in a Riemannian manifold $(M^n,g_0)$ with $n>2$. Suppose that $g$ is a complete metric defined on $\Omega$ which is conformal to $g_0$ and satisfies
\begin{equation}\label{b-ric}
\int_{\Om}|Ric_g|^{\frac{n}{2}}dV_g<\infty.
\end{equation}
Then $M\setminus\Om$ is discrete.

In particular,  if $(M,g_0)$ is a closed Riemannian manifold, then there exists finitely many points $\{p_1,p_2,\cdots,p_N\}$ in $M$ such that
\[\Om=M\setminus\{p_1,p_2,\cdots,p_N\}.\]
\end{thm}

As a corollary, we have
\begin{cor}
If $(M,g)$ is a complete Riemannian manifold with 
\[\int_M|Ric_g|^\frac{n}{2}dV_g<+\infty,\] 
then $(M,g)$
is not conformal to the interior of any compact manifold with boundary.
\end{cor}

The method of  proof of Theorem \ref{main} is quite different from those of \cite{CH,CQY,MQ}.
The arguments in this paper are based on a new observation stated as follows. Let $L_{(x_0, y_0]}=\{(1-t)x_0+ty_0:t\in(0,1]\}$ be a segment in $\R^n$ and
$U\subset \R^n$ be a connected open set. We assume  that $L_{(x_0, y_0]}\subset U$ and $x_0\notin U$. We find that if a conformal metric sequence $\{g_k=u_k^\frac{4}{n-2}g_{euc}\}$ converges to a Ricci flat metric $g_\infty=u_\infty^\frac{4}{n-2}g_{euc}$ in $U$, and if the length of $L_{(x_0,y_0]}$ with respect to the metric $g_k$ is infinity, namely $
\text{Length}(L_{(x_0, y_0]},g_k)=+\infty$, then
$$
u_\infty=\frac{c_0}{|x-x_0|^{n-2}}
$$
in $U$ for some $c_0>0$. This also means that such an $x_0$ is in fact unique. Indeed, if $L_{(x_0',y_0']}\subset U$ is another segment with $x_0'\notin U$ and $\text{Length}(L_{(x_0,y_0]},g_k)=+\infty$, then we also have
$$
u_\infty=\frac{c_0}{|x-x_0'|^{n-2}}
$$
for some $c_0'$. This implies that $x_0=x_0'$.

Theorem \ref{main} will be proved by a contradiction argument. If Theorem \ref{main} is false, then $M\setminus\Omega$ has at least one limit point. Blowing up $g$ near  a limit point in local coordinates, we will get a connected open set $U\subset\R^n$, and a metric sequence $\{g_k\}$ converging to a Ricci flat metric in $U$. Moreover, there are  at least two segments $L_{(x_0,y_0]}$, $L_{(x_0', y_0']}\subset U$, such that   $x_0$, $x_0'\notin U$, $x_0\neq x_0'$, and
$$
\text{Length}(L_{(x_0,y_0]},g_k),\s \text{Length}(L_{(x_0',y_0']},g_k)=+\infty.
$$
So, we obtain a 
contradiction.

Note that we only use the fact that
$Ric_{g_k}$ converges to 0 locally in the above arguments. Recall that the $Q$-curvature on a 4-dimensional manifold is defined as (cf \cite{AC})
\begin{equation}\label{Q.def}
Q_g=-\frac{1}{12}\De_g R_g-\frac{1}{4}|Ric_g|^2+\frac{1}{12}R^2_g.
\end{equation}
Then $R_g=Q_g^{-}=0$ implies that $Ric_g=0$. Hence we have the following result.

\begin{thm}\label{Q-curvature}
Let $(M,g_0)$ be a 4-dimensional Riemannian manifold. Let $\Om$ be a domain in $M$ and $g=u^2g_0$ be a complete metric defined on $\Omega$, which satisfies
$$
\int_\Om(|R_g|^2+|Q_g^{-}|)dV_g<+\infty.
$$
Then $M\setminus\Omega$ is discrete.
\end{thm}

In fact, by combining Theorem \ref{Q-curvature} with  \cite[Corollary 1.5]{LW}, we can show $Ric_g\in L^2(M)$ and $Q_g\in L^1(M)$ when $(M,g_0)$ is closed.

Our paper is organized as follows. In Section \ref{s: pre}, we briefly review some basic notations and results concerning conformal scalar curvature equations, then we classify the Ricci flat metrics
which are conformal to the Euclidean metric in subsection 2.4. Next we study properties of a metric sequence, which tends to
a Ricci flat metric. In the last section we give the proof of the main Theorem \ref{main} and Theorem \ref{Q-curvature}.

\section{Preliminaries}\label{s: pre}

\subsection{Notations}
First, we remark that, throughout the whole paper, we use $B_r(x)$ to denote the Euclidean ball of radius $r$, centered at $x$. When $x=0$, we simply write it as $B_r$. We also denote the Euclidean metric as $g_{euc}$.

Let $g_0$ be a metric defined on $B_1$ and
\[g=e^{\phi}g_0=\rho g_0\]
be a pointwise conformal metric, where $\phi: B_1\to \Real$ is a smooth function and $\rho=e^{\phi}$. Then the Ricci curvature tensor of $g$ is
\begin{equation}\label{Ric}
Ric_g=Ric_{g_0}-\frac{n-2}{2}\mbox{Hess}_{g_0}\phi+\frac{n-2}{4}\n_{g_0}\phi\otimes\n_{g_0}\phi-\frac{1}{2}(\De_{g_0} \phi+\frac{n-2}{2}|\n \phi|^2_{g_0})g_0,
\end{equation}
and the scalar curvature of $g$ has the below form
\begin{equation}\label{Scalar}
R_g=e^{-\phi}(R_{g_0}-(n-1)\De_{g_0} \phi-\frac{(n-1)(n-2)}{4}|\n \phi|^2_{g_0}).
\end{equation}

Taking $\rho=u^{\frac{4}{n-2}}$ for some smooth positive function $u$, $R(g)$ have the following precise formula
\begin{equation}\label{eq-conf}
-\frac{4(n-1)}{n-2}\De_{g_0}u+R_{g_0}u=R_gu^{\frac{n+2}{n-2}}.
\end{equation}

Moreover, it is easy to check that for any $\lambda>0$,
$$
|R_g|^\frac{n}{2}dV_g=|R_{\lambda g}|^\frac{n}{2}dV_{\lambda g},\s and\s
|Ric_g|^\frac{n}{2}dV_g=|Ric_{\lambda g}|^\frac{n}{2}dV_{\lambda g}.
$$

We refer to \cite{LP} for more details.
\medskip

\subsection{$\epsilon$-regularity}

In this subsection, we let $v$ be a weak solution of
\begin{equation}\label{equation.epsilon}
-\mbox{div}(a^{ij}v_{j})=fv,
\end{equation}
where
\begin{equation}\label{aij}
0<\lambda_1\leq a^{ij},\s \|a^{ij}\|_{C^0(B_2)}+\|\nabla a^{ij}\|_{C^0(B_2)}
<\lambda_2.
\end{equation}

We have the following preliminary lemmas.
\begin{lem}\label{Lalpha}
Suppose $v\in W^{1,2}(B_2)$ is a positive weak solution of \eqref{equation.epsilon} and \eqref{aij} holds. We assume
$$
\int_{B_2}|f|^\frac{n}{2}\leq \Lambda.
$$
Then there exists a positive constant $C$ depending only on $\lambda_1$, $\lambda_2$ and $\Lambda$ such that
$$
r^{2-n}\int_{B_r(x)}|\nabla\log v|^2<C,\s \forall B_r(x)\subset B_1.
$$
\end{lem}

\begin{lem}\label{regularity}
Let $v\in W^{1,2}(B_2)$ be a positive solution of \eqref{equation.epsilon}. Assume  \eqref{aij} holds, and $\log v\in W^{1,2}(B_2)$.
Then for any $q\in (0,\frac{n}{2})$, there exists $\epsilon_0
=\epsilon_0(q,\lambda_1,\lambda_2)>0$, such that
if
$$
\int_{B_2}|f|^\frac{n}{2}<\epsilon_0,
$$
then
$$
\|\nabla\log v\|_{W^{1,q}(B_{\frac{1}{2}})}\leq C(\lambda_1,\lambda_2,\epsilon_0),
$$
and
$$
e^{-\frac{1}{|B_\frac{1}{2}|}\int_{B_\frac{1}{2}}\log v}\|v\|_{W^{2,q}(B_\frac{1}{2})}+e^{\frac{1}{|B_\frac{1}{2}|}\int_{B_\frac{1}{2}}\log v}\|v^{-1}\|_{W^{2,q}(B_\frac{1}{2})}
\leq C(\lambda_1,\lambda_2,\epsilon_0).
$$
\end{lem}
For the proofs of the above lemmas, one can refer to
\cite{DLX,X} (or see the proof of \cite[Theorem 4.4]{HL}).

As a corollary, we have the following:
\begin{cor}\label{convergence}
Let $\Omega$ be a connected open set in $\R^n$ and
$g$ be a smooth metric defined on $\Omega$ with
$$
0<\lambda<(g_{ij})<\lambda',\s \|g_{ij}\|_{C^1}<\Lambda.
$$
Let $\hat{g}_k=u_k^\frac{4}{n-2}g_k$, where $g_k$ converges to $g$ in $C^2(\Omega)$. Then there exists $\epsilon_0'>0$ which only depends on $\lambda$, $\lambda'$ and $\Lambda$, such that if $\|R(\hat{g}_k)\|_{L^\frac{n}{2}}<\epsilon_0'$, then there exists $c_k$, such that $c_ku_k$, $\frac{1}{c_ku_k}$ and $\log c_ku_k$ converge weakly in $W^{2,q}_{loc}(\Omega)$
for any $q<\frac{n}{2}$.
\end{cor}

\proof
By a covering argument,
\[\|\nabla \log u_k\|_{L^2(\Omega_1)}<C(\Omega_1)\]
for any $\Omega_1\subset\subset\Omega$. Let $B_\delta(p)\subset\subset\Omega$ and set $c_k$ to be a constant, such that
$$
\int_{B_\delta(p)}\log c_ku_k=0.
$$
By the Poincar\'e inequality (see \cite[Theorem 5.4]{ABM}), $\int_{\Omega_1}|\log c_ku_k|<C(\Omega_1)$.

On the other hand, the formula \eqref{eq-conf} gives
$$
-\Delta_{g_k}u_k=c(n)(-R_{g_k}u_k+R_{\hat{g}_k}u_k^\frac{n+2}{n-2})=:f_ku_k,
$$
and the assumptions of $g_k$ and $R(\hat{g}_k)$ implies there exist positive numbers $K_0$ and $r_0$ such that we have a bound
$$
\int_{B_{2r_0}(x)}|f_k|^\frac{n}{2}\leq C\epsilon_0'< \ep_0
$$
for any $x\in \Om_1$ and $k\geq K_0$.

Then applying Lemma \ref{regularity} on $B_{r_0}(x)$ and using a covering argument, we get the desired result.
\endproof

\medskip

\subsection{Weak convergence of curvatures}
The aim of this subsection is to show that when both $\{g_k\}$ and $g_k^{-1}$ converge
in $\cap_{p\in[1,n)}W^{1,p}$, $Ric_{g_k}$ converges in the sense of distributions. 
For convenience, we let $\mathcal{D}(B_1, \Real^{n})$  denote the space of smooth $n$-dimensional vector-valued functions which are compactly supported in $B_1$, and set $\mathcal{D}(B_1)=\mathcal{D}(B_1, \Real)$.

\begin{lem}\label{curvature.convergence}
Let
$$
g_k=g_{k,ij}dx^i\otimes dx^j,\s g=g_{ij}dx^i\otimes dx^j
$$
be smooth metrics defined on $B_1$. We assume
$g_{k,ij}$ and $g_{k}^{ij}$ converge to
$g_{ij}$ and $g^{ij}$ in $W^{1,p}(B_1)$ for any
$p<n$.
Then for any $n\geq 3$ and $X$, $Y\in \mathcal{D}(B_1,\R^n)$, we have
\begin{equation}\label{conv-curvature}
\lim_{k\rightarrow+\infty}\int_{B_1}Ric_{g_k}(X,Y)dV_{g_k}=
\int_{B_1}Ric_{g}(X,Y)dV_{g}.	
\end{equation}
Moreover, if $\|Ric_{g_k}\|_{L^\frac{n}{2}(B_1,g_k)}\rightarrow 0$, then $Ric_g=0$.
\end{lem}

\proof
We  denote 
$X^{b,g_k}$ and $Y^{b,g_k}$ to be the  dual
of $X$ and $Y$ with respect to metric $g_k$ respectively. The Bochner formula for 1-forms (cf. \cite{P}) gives
$$
\<\Delta_{g_k}X^{b,g_k},Y^{b,g_k}\>_{g_k}=
\<\n^*_{g_k}\n_{g_k}X^{b,g_k},Y^{b,g_k}\>_{g_k}+Ric_{g_k}(X,Y),
$$
where $\De_{g_k}=dd^{*,g_k}+d^{*,g_k}d$, $d^{*,g_k}$  and $\n^*_{g_k}$ are the dual operators of  $d$ and $\n_{g_k}$ with respect to $g_k$ respectively.

Then it follows that
\begin{equation}\label{Bochner-for}
\begin{aligned}
\int_{B_1} Ric_{g_k}(X,Y)dV_{g_k}=&-\int_{B_1} \<\n_{g_k} X, \n_{g_k}Y\>_{g_k}dV_{g_k}\\
&+\int_{B_1} \<dX^{b,g_k}, dY^{b,g_k}\>_{g_k}dV_{g_k}\\
	&+\int_{B_1} \<d^{*,g_k}X^{b,g_k}, d^{*,g_k}Y^{b,g_k}\>_{g_k}dV_{g_k}.
\end{aligned}
\end{equation}

Set $X=X^i\frac{\p}{\p x^i}$ and recall that
$$
\nabla_{g_k}X=(\frac{\partial X^i}{\partial x^j}+\Gamma^i_{jm}(g_k)X^m)dx^j\otimes\frac{\partial}{\partial x^i},
$$
$$
dX^{b,g_k}=\frac{\partial (X^mg_{k, mj})}{\partial x^i}dx^i\wedge dx^j.
$$
$$
d^{*,g_k}X^{b,g_k}=-g^{ij}_k(\frac{\partial (X^mg_{k, mj})}{\partial x^i}-X^mg_{k, ms}\Gamma^s_{ij}(g_k))
$$
and
$$
\Gamma_{ij}^m(g_k)=\frac{1}{2}g_k^{ms}(\frac{\partial g_{k,is}}{\partial x^j}+\frac{\partial g_{k,js}}{\partial x^i}-\frac{\partial g_{k, ij}}{\partial x^s}).
$$
Then we can find a multivariate polynomial $P$, such that
the integrands in the right hand side of \eqref{Bochner-for} are bounded by
$$
\left(1+\sum_{i,j,m}\left|\frac{\partial g_{k,ij}}{\partial x^m}\right|\right)^2\times P(X,Y,\nabla X,\nabla Y,g_k,g_k^{-1}).
$$

Since $g_k$ and $g_k^{-1}$ converge in $W^{1,p}$ for any $p<n$,  it is not difficult to check that 
\begin{align*}
&-\int_{B_1} \<\n_{g_k} X, \n_{g_k}Y\>_{g_k}dV_{g_k}+\int_{B_1} \<dX^{b,g_k}, dY^{b,g_k}\>_{g_k}dV_{g_k}\\
&+\int_{B_1} \<d^{*,g_k}X^{b,g_k}, d^{*,g_k}Y^{b,g_k}\>_{g_k}dV_{g_k}
\\
\to &-\int_{B_1} \<\n_{g} X, \n_{g}Y\>_{g}dV_{g}+\int_{B_1} \<dX^{b,g}, dY^{b,g}\>_{g}dV_{g}\\
&+\int_{B_1} \<d^{*,g}X^{b,g}, d^{*,g}Y^{b,g}\>dV_{g}.
\end{align*}
Hence we get 
\[\lim_{k\rightarrow+\infty}\int_{B_1}Ric_{g_k}(X,Y)dV_{g_k}=
\int_{B_1}Ric_{g}(X,Y)dV_{g}.\]

Furthermore, when
$\|Ric_{g_k}\|_{L^\frac{n}{2}(B_1,g_k)}\rightarrow 0$, we have
\begin{align*}
\left|\int_{B_1}Ric_{g}(X,Y)dV_{g}\right|\leq& \varliminf_{k\to \infty}\norm{Ric_{g_k}}_{L^\frac{n}{2}(B_1,g_{k})}\\
&\times\left(\int_{B_1}\left((g_{k,ij}X^iX^j)(g_{k,ij}Y^iY^j)\right)^\frac{n}{2(n-2)}\sqrt{|g_k|}dx\right)^{1-\frac{2}{n}}\\
=&0.
\end{align*}

Therefore, $Ric_g(X,Y)=0$ for any $X$ and $Y\in\mathcal{D}(B_1,\R^n)$.
So, the proof is completed.\\
\endproof

\begin{lem}\label{Q} 
Let $g_k$, $g$
be as in Lemma \ref{curvature.convergence}.
Additionally, assume $n=4$ and
$$
\int_{B_1}|R_{g_k}|^2dV_{g_k}+\int_{B_1}Q_{g_k}^{-}dV_{g_k}\rightarrow 0.
$$
Then $Ric_{g}=0$.
\end{lem}

\proof
For any nonnegative $\varphi\in\mathcal{D}(B_1)$, we can show
$$
\int_{B_1}\varphi\Delta_{g_k} R_{g_k}dV_{g_k}=
\int_{B_1}R_{g_k}\Delta_{g_k}\varphi dV_{g_k}=
\int_{B_1}R_{g_k}(\sqrt{|g_k|}g^{ij}_k\varphi_i)_jdx
\rightarrow 0,
$$
and
$$
\int_{B_1}Q_{g_k}\varphi dV_{g_k}\geq
-\int_{B_1}Q_{g_k}^{-}\varphi dV_{g_k}\rightarrow 0.
$$
Then, by \eqref{Q.def}, we get $Ric_g=0$.
\endproof

\subsection{Conformal Ricci flat metrics}\label{ricci.flat.cl}

In this section, we classify Ricci flat metrics which are conformal to the Euclidean metric. 
\begin{prop}\label{bble}
Let $U$ be a connected open set in $\Real^{n}$ with $n\geq 3$. Suppose that $g=u^{\frac{4}{n-2}}g_{euc}$ is a smooth metric on $U$. If $Ric_g\equiv0$,
then either $u$ is a constant function or $u=\frac{c_0}{|x-x_0|^{n-2}}$ for some positive constant $c_0$ and $x_0\in \Real^n$.
\end{prop}
\begin{proof}
Since  $(U,g)$ is conformally flat, 
$Ric_g=0$ is  equivalent to the sectional curvature $K_g=0$ here. Hence,  
we can find a domain $V\subset U$
and an embedding $\phi:V\rightarrow\R^n$, such that 
$g=\phi^*(g_{euc})$ in $V$. Obviously, $\phi$ is  conformal from $(V,g)$ to $\R^n$. By  Liouville's
Theorem (see \cite{IM98}), we can extend $\phi$ to 
$\hat{\phi}$, a M\"obious transformation of $\R^n\cup\{\infty\}$. 

Let $U'=U\setminus\{\hat{\phi}^{-1}(\infty)\}$ and define $\hat{g}=(\hat{\phi}|_{U'})^*(g_{euc}):=v^\frac{4}{n-2}g_{euc}$ in $U'$. Since both $R_{\hat{g}}$ and $R_{g}$
are 0 in $U'$, by the equation \eqref{eq-conf},  $u$ and $v$ are harmonic in $U'$. Since $u=v$ in $V$, we get $u\equiv v$ in  $U'$. Then $v$ is smooth in $U$, hence
$\hat{\phi}^{-1}(\infty)\notin U'$, and $U'=U$. 

By Liouville's theorem, we can write $\hat{\phi}(x)=\lambda A(x-x_0)+b$ or
$\hat{\phi}(x)=\lambda \frac{A(x-x_0)}{|A(x-x_0)|^2}+b$, where $\lambda\in \R\setminus\{0\}$, $A\in O(n)$ and $x_0$, $b\in\R^n$. This complete the proof.
\end{proof}

\section{A sequence converges to a Ricci flat metric}
In this section, we study properties of a metric sequence, whose limit is conformal Ricci flat.
We intend to show that in our case, the limit must be an inverted Euclidean metric.

We start with some technical lemmas as follows.
\begin{lem}\label{G}
Let $G(t)$ be one of the following functions:
$$
G_1=\frac{1+(t/2+\mu)^2}{1+(t+\mu)^2},\s
G_2=\frac{(t/2+\nu)^2}{(t+\nu)^2},
$$
where $\mu$ and $\nu$ are given constants and
$\nu>0$. Let $\alpha\in(\frac{1}{2},1)$. If $t>0$ and $G(t)\geq\alpha$, then for any $0<t'<t$ we have $G(t')>\alpha$.
\end{lem}

\begin{proof}
First, we assume $G(t)=\frac{1+(t/2+\mu)^2}{1+(t+\mu)^2}$.
The inequality
\[\frac{1+(\frac{t}{2}+\mu)^2}{1+(t+\mu)^2}\geq\alpha,\]
is equivalent to the below form
\[p(t)=(4\alpha-1)t^2+(8\alpha\mu-4\mu)t+4\alpha+4\alpha\mu^2-4-4\mu^2\leq 0.\]

For the sake of convenience, we write it as
\[p(t)=At^2+4(2\alpha-1)\mu t+B\leq 0,\]
where $A>0$, $B<0$. Let $a_1$ and $a_2$ be the roots of $p=0$. Since $B/A<0$, we may assume $a_1<0<a_2$. Therefore, if $p(t)\leq 0$, then $p(t')<0$.

When $G(t)=\frac{(t/2+\nu)^2}{(t+\nu)^2}$, from the inequality
$$
\frac{(\frac{t}{2}+\nu)^2}{(t+\nu)^2}\geq\alpha,
$$
we get
$$
p_2(t)=(4\alpha-1)t^2+(8\alpha\nu-4\nu)t+4\alpha^2\nu^2-4\nu^2\leq 0.
$$
We write it as
$$
p_2(t)=A't^2+4(2\alpha-1)\mu t+B'\leq 0,
$$
where $A'>0$, $B'< 0$. Similar to the first case, we get the desired result.
\end{proof}

A direct corollary of Lemma \ref{G} is as follows.
\begin{cor}\label{F}
Let $F(t)$ be one of the following two functions:
\[F_1=\frac{\arctan{(\lambda t+\mu)}-\arctan(\frac{\lambda t}{2}+\mu)}{\arctan (\frac{\lambda t}{2}+\mu)-\arctan(\frac{\lambda t}{4}+\mu)},\s
F_2=\frac{\frac{1}{t/2+\nu}-\frac{1}{t+\nu}}{\frac{1}{t/4+\nu}-\frac{1}{t/2+\nu}},
\]
where $\lambda,\nu>0$ and $\mu$ are given constants. Assume $t>0$.
Then $F(t)\geq\frac{3}{2}$ implies that $F(t/2)>\frac{3}{2}$.
\end{cor}

\begin{proof}
The Cauchy mean value theorem implies that there exist $\xi\in (\frac{\lambda}{2}t,\lambda t)$ and $\xi'\in(\frac{\lambda}{4}t,\frac{\lambda}{2}t)$, such that
\[2G_1(\xi)=F_1(t),\quad 2G_1(\xi')=F_1(t/2).\]
Then $F_1(t)\geq \frac{3}{2}$ is equivalent to $G_1(\xi)\geq \frac{3}{4}$. Since $0<\xi'<\xi$, the desired result follows from lemma \ref{G}.

Since the proof of the case of $F=F_2$ is similar, we omit it.
\end{proof}

Before we state the next result, we define some notations.
For a unit vector $v$ and a number $\alpha\in(0,\pi)$, we define
$$
C_\alpha(v)=\{x\in\R^n:\measuredangle(x,v)<\alpha\}.
$$

Let $g$ be a metric defined on $C_\alpha(v)$ with
$$
\|g-g_{euc}\|_{C^2(C_\alpha(v)\cap B_1)}<\tau_1,
$$
and
$g'=u^\frac{4}{n-2}g$ with
$$
\int_{C_\alpha(v)\cap B_1}|Ric_{g'}|^{\frac{n}{2}}dV_{g'}<\tau_2.
$$

We denote $L_k(g')$ to be the length of $\ga=tv$ on the interval $[2^{-k-1},2^{-k}]$ with respect to the metric $g'$.
\begin{lem}\label{cone}
There exists an $\epsilon_1>0$, such that if
$\tau_1,\tau_2<\epsilon_1$, then
\[\frac{L_1(g')}{L_2(g')}\geq \frac{3}{2}\]
implies that
\[\frac{L_{2}(g')}{L_{3}(g')}>\frac{3}{2}.\]
\end{lem}
\begin{proof}
We prove this result by contradiction. Suppose that there exist a sequence of $g_k$ and $g_k'=u_k^\frac{4}{n-2}g_k$, such that
$g_k\rightarrow g_{euc}$ in $C^2(C_\alpha(v)\cap B_1)$,
$\|Ric_{g_k'}\|_{L^\frac{n}{2}(C_\alpha(v)\cap B_1)}\rightarrow 0$
and
\[\frac{L_{1}(g_k')}{L_{2}(g_k')}\geq \frac{3}{2}\quad \text{and}\quad\frac{L_{2}(g_k')}{L_{3}(g_k')}\leq \frac{3}{2}.\]
By Corollary \ref{convergence}, we can find a sequence of constants $\{c_k\}$, such that
$c_ku_k$ converges to a positive function $u_\infty$ weakly in $W^{2,q}_{loc}(C_\al(v)\cap B_1)$ for any $q\in[1,\frac{n}{2})$. It is easy to check that
$u_\infty$ is harmonic, and $g_\infty=u_\infty^\frac{4}{n-2}g_{euc}$ is smooth.
Then, it follows from Lemma \ref{curvature.convergence} that
\[Ric_{g_\infty}\equiv 0.\]
By Proposition \ref{bble}, $u_\infty=constant$ or
\[u_\infty(x)=\frac{c_0}{|x-x_0|^{n-2}}.\]
It is not difficult to see that $x_0\notin C_\al(v)\cap B_1$.

Moreover, the classical trace theorem (cf. \cite{AF}) shows
\[\frac{L_1(g_k')}{L_2(g_k')}\to \frac{\int_{1/2}^{1}u_{\infty}^{\frac{2}{n-2}}(tv)dt}{\int_{1/4}^{1/2}u_{\infty}^{\frac{2}{n-2}}(tv)dt}\geq \frac{3}{2},\]
and
\begin{equation}\label{decay}
\frac{L_2(g_k')}{L_3(g_k')}\to \frac{\int_{1/4}^{1/2}u_{\infty}^{\frac{2}{n-2}}(tv)dt}{\int_{1/8}^{1/4}u_{\infty}^{\frac{2}{n-2}}(tv)dt}\leq \frac{3}{2}.
\end{equation}

Next, we show the contradictions for all the cases of $u_\infty$ could be.
\begin{itemize}
\item[$(1)$] In the case that $u_\infty=constant$, we have
\[\frac{\int_{1/4}^{1/2}u_{\infty}^{\frac{2}{n-2}}(tw)dt}{\int_{1/8}^{1/4}u_{\infty}^{\frac{2}{n-2}}(tw)dt}=2>\frac{3}{2},\]
which is a contradiction.
\item[$(2)$] In the case of $u_\infty(x)=\frac{c_0}{|x-x_0|^{n-2}}$ and $x_0=0$, we have
\[\frac{\int_{1/2}^{1}u_{\infty}^{\frac{2}{n-2}}(tw)dt}{\int_{1/4}^{1/2}u_{\infty}^{\frac{2}{n-2}}(tw)dt}=\frac{1}{2}< \frac{3}{2},\]
which is also a contradiction.
\item[$(3)$] In the case of $x_0\neq 0$ and $0<\measuredangle(v, x_0)<\pi$, by setting $a=|x_0|$ and $\theta=\measuredangle(v,x_0)$, we have $u_\infty(\ga(t))=(a^2+t^2-2at\cos\theta)^{\frac{2-n}{2}}$. Then a simple calculation shows
\[F_1(1)=\frac{\int_{1/2}^{1}u_{\infty}^{\frac{2}{n-2}}(tv)dt}{\int_{1/4}^{1/2}u_{\infty}^{\frac{2}{n-2}}(tv)dt} \quad \text{and}\quad
F_1(1/2)=\frac{\int_{1/4}^{1/2}u_{\infty}^{\frac{2}{n-2}}(tv)dt}{\int_{1/8}^{1/4}u_{\infty}^{\frac{2}{n-2}}(tv)dt}.\]
Here $F_1$ is as in Corollary \ref{F} with
\[\lambda=\frac{1}{a\sin\theta},\quad \mu=-\frac{\cos\theta}{\sin\theta}.\]
Therefore, Corollary \ref{F} implies
\[\frac{\int_{1/4}^{1/2}u_{\infty}^{\frac{2}{n-2}}(tv)dt}{\int_{1/8}^{1/4}u_{\infty}^{\frac{2}{n-2}}(tv)dt}>\frac{3}{2},\]
which contradicts \eqref{decay}.

\item[$(4)$] In the case $x_0\neq 0$ and $\measuredangle(v, x_0)=\pi$, we also set $a=|x_0|>0$, then $u_\infty(\ga(t))=(t+a)^{(2-n)}$
and
\[F_2(1)=\frac{\int_{1/2}^{1}u_{\infty}^{\frac{2}{n-2}}(tv)dt}{\int_{1/4}^{1/2}u_{\infty}^{\frac{2}{n-2}}(tv)dt} \quad \text{and}\quad
F_2(1/2)=\frac{\int_{1/4}^{1/2}u_{\infty}^{\frac{2}{n-2}}(tv)dt}{\int_{1/8}^{1/4}u_{\infty}^{\frac{2}{n-2}}(tv)dt}.\]
Then, Corollary \ref{F} also yields a contradiction in this case.
\end{itemize}
\end{proof}

For simplicity, we denote
$$
V(x,B_r(x'))=\{(1-t)x+ty:t\in(0,1),\s y\in B_r(x')\}.
$$
Now, we state and prove the main result of this section, which will play an essential role in the following section.

\begin{thm}\label{uni-bubble}
Let $\beta\in(0,1)$, 
and assume $x_k\rightarrow x_\infty$. Let $\hat{g}_k=u_k^\frac{4}{n-2}g_k$, where
$g_k$ is a smooth
metric defined on $V(x_k,B_{\beta})$, which
converges to a constant metric $g_c$ smoothly. We suppose that
\begin{itemize}
\item[$(1)$] $u_k$ converges to a positive function $u$ weakly in $W^{2,q}_{loc}(V(x_\infty, B_\beta))$ for any $q<\frac{n}{2}$;
\item[$(2)$] $\|Ric_{\hat{g}_k}\|_{L^\frac{n}{2}(V(x_k,B_\beta),\hat{g}_k)}\rightarrow 0$;
\item[$(3)$] $\text{Length}((1-t)x_k|_{t\in(0,1]},\hat{g}_k)=+\infty$ for any $k$,
\end{itemize}
then
\begin{equation}\label{u.b.}
u(x)=\frac{c_0}{g_c(x-x_\infty,x-x_\infty)^{\frac{n-2}{2}}},\s \mbox{ for some }c_0>0.
\end{equation}
\end{thm}
\begin{proof}
Obviously, it suffices to show \eqref{u.b.} holds
on $B_{\tau\beta}$  any fixed $\tau\in(0,1)$.

Let 
$$
g_k'=\frac{1}{\de_k^2}(\sigma_k^*(g_c^{-1}{g}_k)),\s and\s
\hat{g}_k'=\frac{1}{\de_k^2}(\sigma_k^*(g_c^{-1}\hat{g}_k))=:(u_k')^\frac{4}{n-2}g_k',
$$
where $\delta_k=\frac{|x_\infty|}{|x_k|}$ and $\sigma_k\in O(n)$, which converges to identity mapping and transforms $\delta_kx_k$ to $x_\infty$. It is easy to check that $g_k'$ converges to $g_{euc}$ and $\hat{g}'_k$ satisfies (1)-(3) with $x_k$
replaced by $x_\infty$ and $\beta$ replaced by $\tau\beta$. Moreover, $u_k$ converges to $u(x)=\frac{c_0}{g_c(x-x_\infty,x-x_\infty)^{\frac{n-2}{2}}}
$ on $V(x_\infty, B_\beta)$ if and only if $u_k'$ converges to
$\frac{c_0'}{|x-x_\infty|^{n-2}}$ on $V(x_\infty, B_{\tau\beta})$.

Without loss of generality, we may assume $g_c\equiv g_{euc}$ and $x_k\equiv x_\infty$.
By Lemma \ref{curvature.convergence} and Proposition \ref{bble}, either $u$ is a constant function or $u=\frac{c_0}{|x-y|^{n-2}}$ for some $y\notin V(x_\infty,B_\beta)$.

By Corollary \ref{convergence}, we can find a sequence of constants $\{c_k\}$, such that $c_ku_k$
converges to a function $v$ in $W^{2,q}_{loc}(V(x_\infty,B_{\tau\beta}))$. Since $u_k$
converges to $u$ in $W^{2,q}(B_{\tau\beta})$, we may set $c_k=1$ and
$v=u$.

Suppose $u\neq\frac{c_0}{|x-x_\infty|^{n-2}}$, which implies that $u$ is continuous at $x_\infty$. Put $g_\infty=u^\frac{4}{n-2}g_{euc}$, and
$$
L_m(g)=\text{Length}((1-t)x_\infty|_{[2^{-m-1},2^{-m}]},g).
$$
We have
$$
\lim_{m\rightarrow+\infty}\frac{L_m(g_\infty)}{L_{m+1}(g_\infty)}=2,
$$
since $G(t)\to 1$ as $t\to 0^+$.

By the trace theorem, we may find $k_0$ and $m_0$, such that
$$
\frac{L_{m_0}(\hat{g}_k)}{L_{m_0+1}(\hat{g}_k)}>\frac{3}{2},\s \forall k>k_0.
$$
By using a rescaling argument and applying Lemma \ref{cone}, we get
$$
\frac{L_{m_0+1}(\hat{g}_k)}{L_{m_0+2}(\hat{g}_k)}>\frac{3}{2}
$$
when $k$ is sufficiently large.
Inductively, we obtain
$$
\frac{L_{m_0+i}(\hat{g}_k)}{L_{m_0+i+1}(\hat{g}_k)}>\frac{3}{2}
$$
for any $i>1$,
hence 
$$
\text{Length}((1-t)x_\infty|_{(0,2^{-m_0}]},\hat{g}_k)\leq
\sum_{i=0}^\infty\left(\frac{2}{3}\right)^iL_{m_0}(\hat{g}_k)<+\infty,
$$
which contradicts the assumption $(3)$.
\end{proof}

\section{Proof of main Theorems}

In this section, we give the proof of Theorem \ref{main} and  Theorem \ref{Q-curvature}. First of all, we provide a sufficient condition that guarantees a closed subset to be finite.

\begin{lem}\label{EG}
Let $E$ be a closed subset of $B_1\subset \Real^n$ with $n\geq 3$.
We assume that
$$
0\in E,\s B_\frac{1}{4}\setminus E\neq\emptyset.
$$
We further assume that there does not exist
$x_0$, $x_k$, $y_k$, and $z_k\in B_\frac{1}{2}$, such that
\begin{itemize}
\item[$(1)$] $x_k$, $y_k\in E$, $z_k\notin E$ and $x_k,y_k,z_k\rightarrow x_0$;

\item[$(2)$] there exists $a>1$, such that
$$
\frac{1}{a}r_k<|x_k-z_k|,|y_k-z_k|<ar_k,
$$
where $r_k=|x_k-y_k|$;

\item[$(3)$] there exists $\beta>0$, such that
$$
V(x_k,B_{\beta r_k}(z_k)),V(y_k,B_{\beta r_k}(z_k))\subset B_1\setminus E.
$$
\end{itemize}
Then $E\cap B_\frac{1}{4}$ has finitely many points.
\end{lem}

\proof
We define
\[E_0=\{x\in E|\,\exists \, r>0\,\, \text{such that}\,\, B_r(x)\cap E=\{x\}\}\]
and
$$
E_1=E\setminus E_0.
$$
Obviously, $E_0$ is a countable discrete set,
and $E_1$ is also a closed subset of $B_1$.

We first show $E_0\cap B_\frac{1}{2}$ is a finite set, then prove $E_1\cap B_\frac{1}{4}=\emptyset$.

\medskip
\noindent\emph{Step 1: We prove that $E_0\cap B_\frac{1}{2}$ is finite.} 
If $E_0\cap B_\frac{1}{2}$ is not finite, then $E_0\cap B_\frac{1}{2}$ has at least one limit point $p$
in $\overline{B}_\frac{1}{2}$, namely there exists a sequence $\{x_k\}\subset E_0\cap B_\frac{1}{2}$,
such that
\[\lim_{k\to \infty}x_k=p.\]

Define $r(x): E_0\to \Real^{+}$ as
\[r(x)=\inf_{y\in E\setminus\{x\}}|y-x|.\]
Since $r_k:=r(x_k)\to 0$ and $E\setminus\{x\}$ is closed, there exists a sequence $\{y_k\}$ in $E\setminus\{x_k\}$ such that
\begin{eqnarray}
\label{eq3}r_k=|x_k-y_k|.
\end{eqnarray}
Let $z_k=\frac{x_k+y_k}{2}$, $x_0=p$, and $\beta=\frac{1}{4}$. Then $x_0$, $x_k$, $y_k$ and $z_k$ satisfies (1)-(3). We get a contradiction.\\

\noindent\emph{Step 2: We claim that $E_1\cap B_\frac{1}{4}=\emptyset$.}
We prove this claim by contradiction. Assume there exits $x_0\in \partial (B_1\setminus E_1)\cap B_\frac{1}{4}$.
Choose $\delta<\frac{1}{4}$, such that $B_{\delta}(x_0)\cap E_0=\emptyset$.

We take a point $p_0\in B_\frac{\delta}{4}(x_0)\setminus E_1$.
Since $E_1$ is closed, and $d(p_0,E_1)\leq |p_0-x_0|<\frac{\delta}{4}$,
there exists a point $\hat p_0\in E_1$ such that \[\mbox{dist}(p_0,E_1)=\mbox{dist}(p_0,\hat p_0).\]
Since
$$
\mathrm{dist}(\hat p_0,E_0)\geq \mathrm{dist}(x_0,E_0)-\mathrm{dist}(x_0,\hat p_0)> \frac{1}{4}\delta,
$$
we have
$$
B_\frac{\delta}{4}(\hat p_0)\cap E=B_\frac{\delta}{4}(\hat p_0)\cap E_1.
$$

After changing coordinates, we may assume $\hat{p}_0=0$ and $v=(p_0-\hat{p}_0)/|p_0-\hat{p}_0|=(0,\cdots,0,1)$.
It is easy to check that, we can choose a $\tau\in(0,\frac{\delta}{4})$ such that
\[\overline{C_{\frac{\pi}{4}}(v)}\cap B_{\tau}(0)\cap E_1=\{0\}.\]
By the definition of $E_1$, we can choose a sequence $\{p_k\}$ in $E_1$ such that
\[p_k\to 0,\quad\text{as}\quad k\to \infty.\]
The rest of the proof will be divided into two cases:

\noindent\emph{Case 1: We assume that there exists a sequence $\{p_k\}\subset E_1$, which is contained in $\overline{C_{\frac{3\pi}{4}}(v)}$, and converges to 0.}
In this case, we set $p_k=(p_k',t_k)$, where $p_k'\in\R^{n-1}$ and $t_k\in\R$. We denote $r_k'=|p_k'|$ and define
$$
s_k=\sup\{s:\exists y'\mbox{ such that }(y',s)\in (\overline{B^{n-1}_\frac{r_k'}{2}(p_k')}\times\R)
\cap \overline{C_{\frac{3\pi}{4}}(v,0)\setminus{C}_{\frac{\pi}{4}}(v,0)}\cap E\}.
$$
Assume $y_k=(y_k',s_k)$ is a corresponding point such that the $n$-th coordinate attains the supremum
defined above.

Let $z_k=(p_k',2r_k')$. Then for any $x\in B_\frac{r_k'}{4}(x_k)$, the segment $\overline{xy_k}$ is still in
$\overline{B_\frac{r_k'}{2}^{n-1}(p_k)}\times\R$, hence the segment $\overline{xy_k}\cap E=\{y_k\}$, otherwise it  contradicts the definition of $s_k$. Then, we have
$$
V(y_k,B_\frac{r_k'}{4}(z_k))\cap E=\emptyset.
$$
Let $x_k=0$. It is easy to check that
$$
V(x_k,B_\frac{r_k'}{4}(z_k))\cap E=\emptyset.
$$
On the other hand, since $\frac{r_k'}{2}\leq r_k=|x_k-y_k|\leq 3r_k'$, then
$(x_k, y_k, z_k)$ satisfies (1)-(3) with
$\beta=\frac{1}{12}$ as shown in Figure 1. We get a contradiction.

\begin{center}
\scalebox{0.7}{
\begin{tikzpicture}
	\filldraw[black,fill opacity=0.05](0,0)--(7,-7)--(14,0);
	\draw (0,-7)--(14,-7);
	\draw[->] (7,-11)--(7,0.2);
	\draw (0,0)--(11,-11);
	\draw (3,-11)--(14,0);
	\draw (8.5,-0)--(8.5,-11);
	\draw (11.5,-0)--(11.5,-11);
	
	\draw[line width=1pt] plot[smooth]coordinates{(3,-7.1)(4,-6.5)(7,-7)(10,-8)(10.7,-4.5)};
	
	\draw[line width=1pt] plot[smooth]coordinates{(10.7,-4.5)(11.2,-5)(12,-6.5)};
	
	\draw(10,-1) ellipse (.75 and 0.75);
	
	\filldraw (10,-1) circle (0.05);
	\filldraw (10,-8) circle (0.05);
	\filldraw (10,-7) circle (0.05);
	\filldraw (10.7,-4.5) circle (0.05);
	
	\draw[dashed](8.5,-4.5)--(11.5,-4.5);
	\draw (10.7,-4.5)--(10.78,-0.5);
	\draw (10.7,-4.5)--(9.03,-0.7);
	\draw (7.5,-7) arc (0:45:.5);
	\draw (7.6,-7) arc (0:-45:.6);
	\draw[->] (8.5,-0.1)--(9.5,-0.1);
	\draw[->] (11.5,-0.1)--(10.5,-0.1);
	\draw[dashed] (10,-1)--(10,-8);
	
	
	\draw (7.7,-6.7) node {$\frac{\pi}{4}$};
	\draw (7.8,-7.4) node {$\frac{\pi}{4}$};
	
	\draw (10,-8.4) node[font=\fontsize{11pt}{0}] {$p_k$};
	\draw (10,-7.4) node[font=\fontsize{11pt}{0}] {$(p_k',0)$};
	\draw (12.2,-4.2) node[font=\fontsize{11pt}{0}] {$y_k=(y_k',s_k)$};
	\draw (13.8,-4.8) node[font=\fontsize{11pt}{0}] {$s_k$ attains the supremum};
	\draw (13.6,-5.5) node[font=\fontsize{11pt}{0}] {on $E\cap \overline{B_{\frac{r_k'}{2}}^{n-1}(p_k')}\times \mathbb{R}$};
	\draw (10,-0.1) node[font=\fontsize{11pt}{0}] {$r_k'$};
	\draw (12,-1.1) node[font=\fontsize{11pt}{0}] {$B_\frac{r_k'}{4}((p_k',2r_k'))$};
	\draw (4,-6.3) node[font=\fontsize{11pt}{0}] {$E$};
	\draw (7.2,-0.5) node[font=\fontsize{11pt}{0}] {$v$};
\end{tikzpicture}}

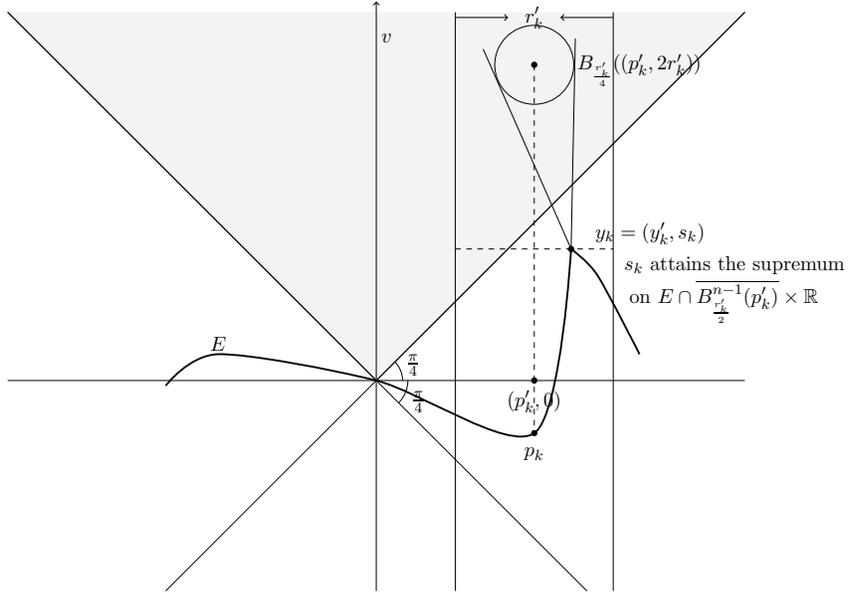
{\bf Figure 1}
\end{center}

\medskip
{\noindent\emph{Case 2: $E_1\cap \overline{C_{\frac{3\pi}{4}}(v)}\cap B_r=\{0\}$ for some $r$.}}
In this case, as shown in Figure 2, we set
$v'=(0,\cdots,0,1,0)$.

For any $x=|x|(\theta^1,\cdots,\theta^n)\notin \overline{C_{\frac{3\pi}{4}}(v)}$, we have
$\theta^n<-\frac{\sqrt{2}}{2}$, which yields that
$\theta^{n-1}\in(-\frac{\sqrt{2}}{2},\frac{\sqrt{2}}{2})$. Therefore, if
$x\notin \overline{C_{\frac{3\pi}{4}}(v)}$, then
$x\in C_\frac{3\pi}{4}(v')\setminus \overline{C_\frac{\pi}{4}(v')}$. Then
$B_r\cap C_{\frac{\pi}{4}}(v')\cap E=\emptyset$
and there exists $p_k'\in C_\frac{3\pi}{4}(v')\cap E$, such that $p_k'\rightarrow 0$. This  has been discussed
in case 1.
\endproof

\begin{center}
\begin{tikzpicture}[scale=0.5]
\draw[dashed] (7,-7)--(3,-3);
\draw[->,dashed] (7,-7)--(0,-7);
\draw (3,-11)--(7,-7)--(11,-11);
\draw[->] (7,-7)--(7,-1);

\draw[line width=1pt] plot[smooth]coordinates{(3,-11.1)(5,-10)(6,-9)(7,-7)};
\draw[line width=1pt] plot[smooth]coordinates{(7,-7)(8,-9)(9,-10)(10,-11)};

\draw (5,-11) node {$E$};
\draw (1,-6.5) node {$v'$};
\draw (7.5,-2) node {$v$};
\end{tikzpicture}

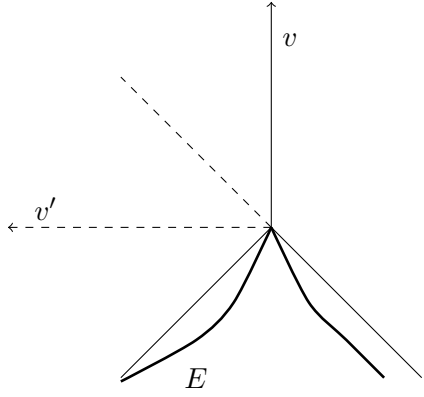
{\bf Figure 2}\\
\end{center}

Now, we are ready to prove our main theorems.

\begin{proof}[\bf{Proof of Theorem \ref{main}}]We only need to
prove Theorem \ref{main} in local coordinates of $M$ around a point $p_0\in \partial\Omega$.
Without loss of generality, we assume $B_1$ lies in the coordinate chart and $p_0=0$,  and set $E=(M\setminus \Om)\cap B_1$.

Suppose $E\cap B_\frac{1}{4}$ is an infinite set.
By Lemma \ref{EG}, we can find $x_0$ $x_k$, $y_k$, and $z_k\in B_\frac{1}{2}$, satifsying (1)-(3) in Lemma \ref{EG}.

We set
$$
g_k=g_{ij}(z_k+r_kx)dx^i\otimes dx^j,\s
u_k(x)=r_k^\frac{n-2}{2}u_k(z_k+r_kx),\s
g_k'=u_k^\frac{4}{n-2}g_k.
$$
Then
$g_k$ converges to a constant metric $g_c$ and
$\|Ric_{g_k'}\|_{L^\frac{n}{2}}\rightarrow 0$.

Assume that
$$
x_k'=(x_k-z_k)/r_k\rightarrow x_\infty,\s
y_k'=(y_k-z_k)/r_k\rightarrow y_\infty,
$$
as $k\to \infty$. Since $|x_k'-y_k'|=1$, $x_\infty\neq y_\infty$.

By Corollary \ref{convergence}, there exist a sequence of constants $\{c_k\}$ such that $c_ku_k$ weakly converges to a positive function $u_\infty$ in
 $W^{2,q}_{loc}(V(x_\infty, B_\beta)\cup V(y_\infty, B_\beta))$ for any $q\in[1, \frac{n}{2})$.

Hence by using
Theorem \ref{uni-bubble} both
on $V(x_k',B_\beta)$ and on $V(y_k',B_\beta)$, we get
$$
u_\infty=
\frac{c_0}{g_c(x-x_\infty,x-x_\infty)^\frac{n-2}{2}}=
\frac{c^\prime_0}{g_c(x-y_\infty,x-y_\infty)^\frac{n-2}{2}}
$$
for some constants $c_0$ and $c^\prime_0$,
which leads to a contradiction with $x_\infty\neq y_\infty$.\\
\end{proof}

By Lemma \ref{Q}, when $n=4$, Theorem \ref{uni-bubble} still holds if we replace (2)
with
$$
\int_{V(x_k,B_\beta)}(|R_{\hat{g}_k}|^2dV_g+|Q_{\hat{g}_k}^{-}|)dV_{\hat{g}_k}\rightarrow 0.
$$
Then the proof of Theorem \ref{Q-curvature} is almost the same as
that of Theorem \ref{main}, hence we omit it.


\begin{center}{\bf  Acknowledgements}\end{center}
 
 We are grateful to the referee for  helpful corrections and highly constructive suggestions which very substantially improved 
the exposition.

\end{document}